\documentclass[11pt,leqno]{amsart}
\usepackage{amsmath,amssymb,amsthm}
\newtheorem{theorem}{Theorem}[section]
\newtheorem{lemma}[theorem]{Lemma}

\newtheorem{proposition}[theorem]{Proposition}
\newtheorem{corollary}[theorem]{Corollary}
\theoremstyle{definition}

\theoremstyle{remark}
\newtheorem{remark}[theorem]{Remark}
\numberwithin{equation}{section}

\def\fnote#1{\footnote}

\def\natu{{\mathbb N}}

\def\real{{\mathbb R}}

\def\ignora#1{}
\def\lbl#1{\label{#1}}        
\def\n3#1{\left\vert  \! \left\vert \! \left\vert \, #1 \, \right\vert \!
  \right\vert \! \right\vert }

\begin{document}

\title{Banach spaces with many boundedly complete basic sequences failing PCP}
\author{Gin\'{e}s L\'{o}pez P\'{e}rez}
\address{Universidad de Granada, Facultad de Ciencias.
Departamento de An\'{a}lisis Matem\'{a}tico, 18071-Granada
(Spain)} \email{glopezp@ugr.es}

\thanks{Partially supported by MEC (Spain) Grant MTM2006-04837 and Junta de Andaluc\'{\i}a Grants FQM-185 and Proyecto
de Excelencia P06-FQM-01438.}  \maketitle \markboth{Gin\'{e}s
L\'{o}pez P\'{e}rez}{Boundedly complete sequences and PCP}
$$\parbox{3cm}

To\ my\ mother\ Francisca\ and\ my\ sister\ Isabel,\ in\ memoriam
$$
\begin{abstract}

We prove that there exist Banach spaces not containing $\ell_1$,
failing the point of continuity property and satisfying that every
semi-normalized basic sequence has a boundedly complete basic
subsequence. This answers in the negative the problem of the
Remark 2 in \cite{R1}.

\end{abstract}

\section{Introduction}
\par
\bigskip

Recall that a Banach space is said to have the point of continuity
property (PCP) provided every non-empty closed and bounded subset
admits a point of continuity of the identity map from the weak to
norm topologies. It is known that Banach spaces with Radon-Nikodym
property, including separable dual spaces, satisfy PCP, but the
converse is false (see \cite{BR}). The PCP has been characterized
for separable Banach spaces in \cite{BR} and \cite{GM}, and this
characterization implies that Banach spaces with PCP have many
boundedly complete basic sequences, and so many subspaces which
are separable dual spaces. As PCP is separably determined
\cite{B}, that is, a Banach space satisfies PCP if every separable
subspace has PCP, it is natural looking for a sequential
characterization of PCP. In this sense, it has been proved in
\cite{R1} that every semi-normalized basic sequence in a Banach
space with PCP has a boundedly complete subsequence. The converse
of the above result is false in general, but it is open for Banach
spaces not containing $\ell_1$ (see Remark 2 in \cite{R1}). The
goal of this note is to prove in corollary \ref{fin} that there
exist Banach spaces failing PCP and not containing $\ell_1$ such
that e\-ver\-y semi-normalized basic sequence has a boundedly
complete subsequence. Concretely, the space $B_\infty$, the
natural predual of the space $JT_\infty$, constructed in \cite{GM}
is the desired example.

We begin with some notation and preliminaries. Let $X$ be a Banach
space and let $\{e_n\}$ be a basic sequence in $X$. $\{e_n\}$ is
said to be semi-normalized if $0<\inf_n\Vert e_n\Vert\leq
\sup_n\Vert e_n\Vert <+\infty$, $X^*$ denotes the topological dual
of $X$ and the closed linear span of $\{e_n\}$ is denoted by
$[e_n]$. $\{e_n\}$ is called\begin{enumerate} \item[i)] {\it
boundedly complete} provided whenever scalars $\{\lambda_i\}$
satisfy\break $\sup_n\Vert\sum_{i=1}^n\lambda_ie_i\Vert<+\infty$,
then $\sum_n\lambda_ne_n$ converges. \item[ii)] {\it shrinking} if
the scalar sequence $\{\Vert f_{\mid [e_n, e_{n+1},
\ldots]}\Vert\}$ converges to zero $\forall f\in X^*$. \item[iii)]
{\it supershrinking} provided $\{e_n\}$ is shrinking and whenever
scalars $\{\lambda_i\}$ satisfy
$\sup_n\Vert\sum_{i=1}^n\lambda_ie_i\Vert<+\infty$ and
$\{\lambda_i\}\to  0$, then $\sum_n\lambda_ne_n$ converges.
\item[iv)] {\it strongly summing} provided is a weakly Cauchy
sequence and when\-ever scalars $\{\lambda_i\}$ satisfy
$\sup_n\Vert\sum_{i=1}^n\lambda_ie_i\Vert<+\infty$, then
$\sum_n\lambda_n$ converges. \end{enumerate}

A boundedly complete basic sequence spans a dual space and a
shrinking basic sequence $\{e_n\}$ spans a subspace whose dual has
a basis $\{f_n\}$, called the sequence of associated functionals
to $\{e_n\}$. A boundedly complete and shrinking basic sequence
spans a reflexive subspace and a basic sequence in a reflexive
space is  both boundedly complete and shrinking \cite {LZ}.

The supershrinking basic sequences appear in \cite{G1} and
\cite{G2}, where it is proved that a Banach space $X$ with a
supershrinking basis not containing $c_0$ is somewhat order one
quasireflexive, whenever $X$ not contains isomorphic subspaces to
$c_0$. Then $X$ has many boundedly complete basic sequences. The
space $B_\infty$ has a supershrinking basis (see \cite{G1} and
theorem IV.2 in \cite{GM}), not contains $c_0$ and fails PCP
\cite{GM}, so $B_{\infty}$ is a good candidate to be the desired
example. Other examples with a supershrinking basis are $c_0$ and
$B$, the natural predual of James tree space $JT$ \cite{GM}. It is
worth to mention that a semi-normalized basis of a Banach space
$X$ is supershrinking if and only if
\begin{equation}\lbl{igualdad} \{x^{**}\in
X^{**}:\lim_nx^{**}(f_n)=0\}=X \end{equation}
 where $\{f_n\}$ is
the associated functional sequence \cite{G1}.

The strongly summing basic sequences appear in \cite{R2}, where it
is proved the remarkable $c_0-$theorem, which assures that every
weak Cauchy non-trivial sequence in a Banach space not containing
$c_0$, has a strongly summing basic subsequence. A weak Cauchy
sequence in a Banach space is said to be non-trivial if does not
converge weakly. Finally, we recall that if $\{e_n\}$ is a
strongly summing sequence, then $\{v_n\}$ is a basic sequence,
where $\{v_n\}$ is the diference sequence of $\{e_n\}$, that is,
$v_1=e_1$ and $v_n=e_n-e_{n-1}$ for $n>1$ (\cite{R2}).

There is a very easy connection between supershrinking, strongly
summing and boundedly complete basic sequences, which implicitly
appears in \cite{R1}. We give it here for sake of completeness.

\begin{lemma}\lbl{diferencias} Let $\{e_n\}$ a semi-normalized strongly summning
basic sequence with diference sequence $\{v_n\}$. If $\{v_n\}$ is
supershrinking, then $\{e_n\}$ is boun\-ded\-ly complete. In fact,
$[e_n]$ is order one quasireflexive, that is, $[e_n]$ has
codimension $1$ in $[e_n]^{**}$\end{lemma}

\begin{proof} Let $\{\lambda_n\}$ be scalars so that
$\sup_n\Vert\sum_{i=1}^n\lambda_1e_i\Vert<+\infty$. We have to
prove that $\sum_n\lambda_ne_n$ converges in order to obtain that
$\{e_n\}$ is boundedly complete. As $\{e_n\}$ is strongly summing,
hence $\sum_n\lambda_n$ converges. Define
$\mu_n=\sum_{i=n}^{+\infty}\lambda_i$ $\forall n$. Then
$\{\mu_n\}$ converges to zero and
\begin{equation}\lbl{diferencia}
\sum_{i=1}^n\mu_iv_i=\sum_{i=1}^{n-1}\lambda_ie_i+\mu_ne_n\
\forall n\in \natu
\end{equation}
So, $\sup_n\Vert\sum_{i=1}^n\mu_iv_i\Vert<+\infty$ and then
$\sum_n\mu_nv_n$ converges, by hypothesis. Finally,
$\sum_n\lambda_ne_n$ converges by \ref{diferencia}, since
$\{\mu_n\}\to 0$.

Now, we conclude that $[e_n]$ is order one quasireflexive. For
this, put $e_n^*=v_n^*-v_{n+1}^*$, where $\{v_n^*\}$ is the
associated functional sequence to $\{v_n\}$. Then $\{e_n^*\}$ is
the associated functional sequence to $\{e_n\}$. Observe that
$[e_n]^*=[v_n^*]$, since $\{v_n\}$ is shrinking. Hence, $[e_n^*]$
has codimension $1$ in $[e_n]^*$, since $x^{**}(e_n^*)=0$ for
every $n$ and $x^{**}(v_1^*)=1$, where $x^{**}(x^*)=\lim_n
x^*(e_n)$ for every $x^*\in [e_n]^*$ exists because $\{e_n\}$ is
weakly Cauchy. In fact, $[e_n]^*=[e_n^*]\oplus [v_1^*]$. But
$[e_n^*]^*$ is canonically isomorphic to $[e_n]$, since $\{e_n\}$
is a boundedly complete sequence. Then $[e_n]$ has codimension $1$
in $[e_n]^{**}$.\end{proof}

\section{Main results}
\par
\bigskip

The corollary \ref{fin} announced in the introduction will be
deduced from the following more general result.

\begin{theorem}\lbl{teorema} Let $X$ be a Banach space with a
semi-normalized supershrinking basis, not containing $c_0$. Then
every non-trivial weak Cauchy sequence has a boundedly complete
basic subsequence.
\end{theorem}

\bigskip

Before prove this theorem, we need the following stability
property of supershrinking basic block sequences.

\begin{lemma}\lbl{bloque} Let $X$ be a Banach space with a semi-normalized supershrinking basis
$\{e_{n}\}$. If
$v_{n}=\sum_{k=\sigma(n-1)+1}^{\sigma(n)}\lambda_{k}e_{k}$ is a
basic block of $\{e_{n}\}$ with $\{\lambda_{n}\}$ bounded, then
$\{v_n\}$ is a supershrinking basic sequence.\end{lemma}

\begin{proof} Let $\{f_n\}$, $\{g_{n}\}$ be the sequences of associated functionals
to $\{e_n\}$ and $\{v_{n}\}$, respectively. Then
$f_{k}=\lambda_{k}g_{n}$ whenever $\sigma(n-1)+1\leq k\leq
\sigma(n)$. In order to show that $\{v_n\}$ is a supershrinking
basic sequence we check the equality \ref{igualdad}.

Pick $y^{**}\in [v_n]^{**}$ with $\lim_{n}y^{**}(g_{n})=0$ then
$\lim_{n}y^{**}(f_{n})=0$, since $\{\lambda_{n}\}$ is bounded. So,
$y^{**}\in X$ and $\{v_n\}$ is supershrinking.\end{proof}

\bigskip

Now, we show that Banach spaces with a supershrinking basis
without copies of $c_0$ contain many reflexive subspaces.

\begin{proposition}\lbl{elton} Let $X$ be a Banach space with a
semi-normalized supershrinking basis $\{e_{n}\}$ without
isomorphic subspaces to $c_{0}$. Then every subsequence of
$\{e_{n}\}$ has a further subsequence whose closed linear span is
a reflexive subspace.
\end{proposition}

\begin{proof}
It is clear that it is enough to prove that $\{e_{n}\}$ has a
subsequence whose closed linear span is a reflexive subspace.

For this, we apply the Elton Theorem \cite{D} to obtain
$\{e_{\sigma(n)}\}$ a basic subsequence of $\{e_{n}\}$ such that
$$\lim_{k}\Vert \sum_{i=1}^{k}a_{i}e_{\sigma(i)}\Vert =+\infty\
\forall \{a_{i}\}\notin c_{0}.$$ We put $Y=[e_{\sigma(n)}]$. To
see that $Y$ is reflexive it suffices to prove that
$\{e_{\sigma(n)}\}$ is a boundedly complete basic sequence in $Y$,
since $\{e_{\sigma(n)}\}$ is a shrinking basic sequence.

Let $\{\lambda_{n}\}\subset \real$ such that $\sup_{n}\Vert
\sum_{k=1}^{n}\lambda_{k}e_{\sigma(k)}\Vert <+\infty$. Then
$\{\lambda_{n}\}\in c_{0}$ and $\sum_{n}\lambda_{n}e_{\sigma(n)}$
converges, since $\{e_{\sigma(n)}\}$ is supershrinking, that is
$Y$ is reflexive.\end{proof}

\bigskip

\begin{proof} {\it of theorem} \ref{teorema}. Let $\{f_n\}$ be the functional sequence associated to $\{e_n\}$ and
assume, without loss of generality that $\{e_n\}$ is monotone,
that is, $\Vert Q_n\Vert\leq 1$ $\forall n\in \natu$, where
$\{Q_n=\sum_{k=1}^nf_k\}$ is the sequence of the projections of
the basis $\{e_n\}$. Put $M=\sup_n\Vert e_n\Vert$ and let
$\{x_n\}$ be a non-trivial weak Cauchy in $X$. By the
$c_0$-theorem, we can assume that there is a strongly summing
basic subsequence of $\{x_n\}$, so we in fact assume that
$\{x_n\}$ itself is a non-trivial weak Cauchy strongly summing
basic sequence.

We claim that there exist integers $0<\sigma(1)<\sigma(2)<\ldots$,
$0=m_0<1=m_1<m_2<\ldots$ and $\{v_n\}$ a basic sequence such
that\begin{enumerate} \item[i)]
\begin{equation}
\vert f_h(x_{\sigma(n)})-f_h(x_k)\vert<\frac{1}{2^{n+3}m_{n}M}\
\forall k\geq \sigma(n),\ h\leq m_{n},\ n\in \natu
\end{equation}
\item[ii)]$v_n\in [e_k:m_{n-1}+1\leq k\leq m_{n+1}]\ \forall n\in
\natu$ \item[iii)] $\Vert v_n-z_n\Vert<1/2^{n+1} \forall n\in
\natu$,\end{enumerate} where $\{z_n\}$ is the diference sequence
of $\{x_{\sigma(n)}\}$, that is, $z_1=x_{\sigma(1)}$,\break
$z_n=x_{\sigma(n)}-x_{\sigma(n-1)}$ for all $n>1$.

As $\{x_n\}$ is weakly Cauchy, there is $\sigma(1)\in \natu$ such
that
\begin{equation}\lbl{1}\vert f_1(x_{\sigma(1)})-f_1(x_k)\vert <1/2^4M\
\forall k\geq \sigma(1).\end{equation}
 Choose $m_2>m_1$ such that
$\Vert \sum_{n=m_2+1}^{+\infty}f_n(x_{\sigma(1)})e_n\Vert<1/2^2$
and put\break $v_1=\sum_{n=1}^{m_2}f_n(x_{\sigma(1)})e_n$. Then
$\Vert z_1-v_1\Vert=\Vert
\sum_{n=m_2+1}^{+\infty}f_n(x_{\sigma(1)})e_n\Vert<1/2^2$.

Pick now $\sigma(2)>\sigma(1)$ such that
\begin{equation}\lbl{2}
\vert f_h(x_{\sigma(2)})-f_h(x_k)\vert<\frac{1}{2^5m_2M}\ \forall
k\geq\sigma(2),\ h\leq m_2 \end{equation} Chose $m_3>m_2$ such
that $\Vert
\sum_{n=m_3+1}^{+\infty}(f_n(x_{\sigma(2)})-f_n(x_{\sigma(1)}))e_n\Vert<1/2^4$.

Put now
$v_2=\sum_{n=m_1+1}^{m_3}(f_n(x_{\sigma(2)})-f_n(x_{\sigma(1)}))e_n$.
Then $\Vert z_2-v_2\Vert\leq \Vert
(f_1(x_{\sigma(2)})-f_1(x_{\sigma(1)}))e_1\Vert
+\Vert\sum_{n=m_3+1}^{+\infty}(f_n(x_{\sigma(2)})-f_n(x_{\sigma(1)}))e_n\Vert<1/2^4+1/2^4=1/2^3$,
by \ref{1} and \ref{2}.

Assume, inductively, that $m_2<m_3<\ldots<m_{n+1}
$,
$\sigma(2)<\sigma(3)<\ldots<\sigma(n)$, $v_1,v_2,\ldots,v_n$ have
been constructed such that
\begin{equation}\lbl{3}
\vert f_h(x_{\sigma(n)})-f_h(x_k)\vert<\frac{1}{2^{n+3}m_{n}M}\
\forall k\geq \sigma(n),\ h\leq m_{n}
\end{equation}
Pick now $m_{n+2}>m_{n+1}$ such that
\begin{equation}\lbl{4}
\Vert\sum_{n=m_{n+2}+1}^{+\infty}(f_n(x_{\sigma(n+1)})-f_n(x_{\sigma(n)}))e_n\Vert<1/2^{n+3}.
\end{equation}
Put
$v_{n+1}=\sum_{i=m_n+1}^{m_{n+2}}(f_i(x_{\sigma(n+1)})-f_i(x_{\sigma(n)}))e_i$.
Then $\Vert z_{n+1}-v_{n+1}\Vert\leq
\Vert\sum_{i=1}^{m_n}(f_i(x_{\sigma(n+1)})-f_i(x_{\sigma(n)}))e_i\Vert+
\Vert\sum_{i=m_{n+2}+1}^{+\infty}(f_i(x_{\sigma(n+1)})-f_i(x_{\sigma(n)}))e_i\Vert<1/2^{n+3}+1/2^{n+3}=1/2^{n+2}$,
by \ref{3} and \ref{4}.

Now, choose $\sigma(n+1)>\sigma(n)$ such that
\begin{equation}
\vert
f_h(x_{\sigma(n+1)})-f_h(x_k)\vert<\frac{1}{2^{n+4}m_{n+1}M}\
\forall k\geq \sigma(n+1),\ h\leq m_{n+1}.
\end{equation}
Then the induction is complete and the claim is proved.

From the claim, it is clear that $\{v_n\}$ is a basic sequence
equivalent to $\{z_n\}$, the diference sequence of
$\{x_{\sigma(n)}\}$, since $\sum_{n=1}^{+\infty}\Vert
z_n-v_n\Vert<1/2$ (see proposition 1.a.9 in \cite{LZ}). Also, we
obtain that
$[v_n,v_{n+1},\ldots]\subset[e_{m_{n-1}+1},e_{m_{n-1}+2},\ldots]$
$\forall n\in \natu$, so $\{v_n\}$ is a shrinking basic sequence,
since $\{e_n\}$ it is.

Now, let us see that $\{v_n\}$ is a supershrinking basic sequence.
For this, we chose $\{\lambda_n\}$ a scalar sequence such that
$\sup_n\Vert\sum_{k=1}^n\lambda_kv_k\Vert<+\infty$ and we have to
prove that $\sum_n\lambda_nv_n$ converges, whenever
$\{\lambda_n\}\to 0$.

From the proof of the claim
$v_1=\sum_{n=1}^{m_2}f_n(x_{\sigma(1)})e_n$, and for every $n>1$
$v_n=\sum_{k=m_{n-1}+1}^{m_{n+1}}(f_k(x_{\sigma(n)})-f_k(x_{\sigma(n-1)}))e_k$.

Put $\mu_i=\lambda_1f_i(x_{\sigma(1)})$ for $1\leq i\leq m_1$,
$\mu_i=\lambda_1f_i(x_{\sigma(1)})+\lambda_2(f_i(x_{\sigma(2)})-f_i(x_{\sigma(1)}))$
for $m_1+1\leq i\leq m_2$ and
$\mu_i=\lambda_{k-1}(f_i(x_{\sigma(k-1)})-f_i(x_{\sigma(k-2)}))+\lambda_k(f_i(x_{\sigma(k)})-f_i(x_{\sigma(k-1)}))$
for $m_{k-1}+1\leq i\leq m_k$ and $k>2$.

As $\{\lambda_n\}\to 0$, $\{e_n\}$ is a seminormalized basis of
$X$ and $\{x_n\}$ is bounded, we deduce that $\{\mu_n\}\to 0$.
Furthermore, we have the following equality for all $n\in \natu$:
\begin{equation}
\sum_{k=1}^{n}\lambda_kv_k=\sum_{k=1}^{m_n}\mu_ke_k+\sum_{k=m_n+1}^{m_{n+1}}\lambda_n(f_k(x_{\sigma(n)})-f_k(x_{\sigma(n-1)}))e_k
\end{equation}
Hence, whenever $m_n+1\leq p<m_{n+1},\ n>1$ we have
\begin{equation}\lbl{chorizo}
 \sum_{k=1}^p\mu_ke_k=\sum_{k=1}^{n}\lambda_kv_k+\sum_{k=m_n+1}^{p}\lambda_{n+1}(f_k(x_{\sigma(n+1)})-f_k(x_{\sigma(n)}))e_k
\end{equation}
$$-\sum_{k=p+1}^{m_{n+1}}\lambda_{n}(f_k(x_{\sigma(n)})-f_k(x_{\sigma(n-1)}))e_k$$

Now, as $\{x_n\}$ and $\{Q_n\}$ are bounded and $\{\lambda_n\}\to
0$, we obtain that
\begin{equation}\lbl{cero}
\lim_n\sum_{k=p+1}^{m_{n+1}}\lambda_{n}(f_k(x_{\sigma(n)})-f_k(x_{\sigma(n-1)}))e_k=
\end{equation}
$$\lim_n\sum_{k=m_n+1}^{p}\lambda_{n+1}(f_k(x_{\sigma(n+1)})-f_k(x_{\sigma(n)}))e_k=0,$$
since for every $m_n+1\leq p<m_{n+1}$, $n\in \natu,\ n>1$ we have:
\begin{equation}\lbl{proj}\sum_{k=p+1}^{m_{n+1}}\lambda_{n}(f_k(x_{\sigma(n)})-f_k(x_{\sigma(n-1)}))e_k
=\lambda_n(Q_{m_{n+1}}-Q_p)(x_{\sigma(n)}-x_{\sigma(n-1)}),\end{equation}
$$\sum_{k=m_n+1}^{p}\lambda_{n+1}(f_k(x_{\sigma(n+1)})-f_k(x_{\sigma(n)}))e_k=
\lambda_{n+1}(Q_p-Q_{m_n})(x_{\sigma(n+1)}-x_{\sigma(n)})$$ From
\ref{chorizo} and \ref{proj}, it can be deduced that
$\sup_p\Vert\sum_{n=1}^p\mu_ne_n\Vert<+\infty$ and so,
$\sum_n\mu_ne_n$ converges, since $\{\mu_n\}\to 0$ and $\{e_n\}$
is supershrinking. Then $\sum_n\lambda_nv_n$ converges by
\ref{chorizo} and \ref{cero} and we have proved that $\{v_n\}$ is
a supershrinking basic sequence equivalent to the diference
sequence of $\{x_{\sigma(n)}\}$. Finally, $\{x_{\sigma(n)}\}$ is
boundedly complete by lemma \ref{diferencias}, since it is
strongly summing. In fact, $[x_{\sigma(n)}]$ is order one
quasireflexive, by lemma \ref{diferencias}.\end{proof}

\bigskip

\begin{corollary}\lbl{fin} Let $X$ be a Banach space with a semi-normalized
supershrinking basis not containing $c_0$. Then every
semi-normalized basic sequence in $X$ has a boundedly complete
subsequence spanning a reflexive or an order one quasireflexive
subspace of $X$.
\end{corollary}

\begin{proof} Let $\{x_n\}$ a semi-normalized basic sequence in
$X$. As $X$ not contains isomorphic subspaces to $\ell_1$, we can
assume that $\{x_n\}$ itself is weakly Cauchy, by the
$\ell_1-$theorem \cite{R3}. If $\{x_n\}$ is not weakly convergent,
then $\{x_n\}$ is a semi-normalized non-trivial weak Cauchy
sequence and $\{x_n\}$ has a boundedly complete subsequence
spanning an order one quasireflesive subspace, by theorem
\ref{teorema}, and we are done.

If $\{x_n\}$ is weakly convergent, then $\{x_n\}$ converges weakly
to zero, because $\{x_n\}$ is a basic sequence. Now, it is
straightforward construct a subsequence of $\{x_n\}$ equivalent to
a basic block of the basis. So, we can assume that $\{x_n\}$ is a
semi-normalized basic sequence equivalent to a basic block of the
basis. Following the proof of proposition 1.a.11 in \cite{LZ}, it
t is easy to construct this basic block satisfying the hypothesis
of lemma \ref{bloque}. Then $\{x_n\}$ is a supershrinking basic
subsequence and, by proposition \ref{elton}, $\{x_n\}$ has a
boundedly complete subsequence spanning a reflexive subspace, so
we are done.\end{proof}

\bigskip

As we announced in the introduction, it is enough apply the
corollary \ref{fin} to obtain the following

\begin{corollary}\lbl{qed} $B_{\infty}$ fails PCP, not contains isomorphic subspaces to $\ell_1$ and every semi-normalized
basic sequence in $B_{\infty}$ has a boundedly complete
subsequence spanning a reflexive or an order one quasireflexive
subspace.
\end{corollary}

\begin{proof} The fact that $B_{\infty}$ has a semi-normalized supershrinking basis is
consequence of theorem IV.2 in \cite{GM}. So $B_{\infty}$ has
separable dual and not contains subspaces isomorphic to $\ell_1$.
Now, $B_{\infty}$ fails PCP and not contains subspaces isomorphic
to $c_0$ \cite{GM}. Finally, by corollary \ref{fin}, every
semi-normalized basic sequence in $B_{\infty}$ has a boundedly
complete subsequence spanning a reflexive or an order one
quasireflexive subspace of $B_{\infty}$.\end{proof}

\bigskip

Let $B$ be the natural predual of James tree space $JT$. It is
known that $B$ satisfies PCP, and also $B$ has a semi-normalized
supershrinking basis. (See \cite{LS} and \cite{G2}). As $B$ not
contains isomorphic subspaces to $c_0$, \cite{LS}, we can apply
the corollary \ref{fin}, as in corollary \ref{qed}, to obtain the
following

\begin{corollary} Every semi-normalized basic sequence in $B$ has
a boundedly complete subsequence spanning a reflexive or an order
one quasireflexive subspace of $B$.\end{corollary}

\bigskip

\begin{remark}i) It has been proved in \cite{DF} that a Banach space $X$ with
se\-pa\-rable dual satisfies PCP if, and only if, every weakly
null tree in the unit sphere of $X$ has a boundedly complete
branch, which can be easily deduced from \cite{G2}. Also, it is
shown in \cite{DF} that this characterization of PCP is not true
for sequences, by proving that every weakly null sequence in the
unit sphere of  $B_{\infty}$ has a boundedly complete subsequence,
while $B_{\infty}$ fails PCP. Hence the corollary \ref{qed}
improves this result, since every weakly null sequence in the unit
sphere of a Banach space has a semi-normalized basic subsequence.

ii) From corollary \ref{fin} one might think that the  good
sequential property in order to imply PCP for Banach spaces with
separable dual is that e\-ve\-ry semi-normalized basic sequence
has a subsequence spanning a reflexive subspace. And this is true,
but this property implies reflexivity. Indeed, assume that $X$ is
a Banach space satisfying that every semi-normalized basic
sequence has a subsequence spanning a reflexive subspace. Take a
bounded sequence $\{x_n\}$ in $X$ and prove that $\{x_n\}$ has a
weakly convergent subsequence. As $X$ not contains subspaces
isomorphic to $\ell_1$, then $\{x_n\}$ has a weak Cauchy
subsequence $\{y_n\}$, by the $\ell_1-$theorem. If $\{y_n\}$ is
not semi-normalized, then $\{y_n\}$ and so $\{x_n\}$ has a
subsequence weakly convergent to zero and we are done. Hence,
assume that $\{y_n\}$ is a semi-normalized weak Cauchy sequence in
$X$. If $\{y_n\}$ is not weakly convergent, then, by the
$c_0-$theorem, for example, $\{y_n\}$ has a semi-normalized basic
subsequence, since $X$ not contains isomorphic subspaces to $c_0$.
By hypo\-the\-sis, $\{y_n\}$ has a subsequence spanning a
reflexive subspace and hence, this subsequence is weakly
convergent to zero, so $\{x_n\}$ has a weakly convergent
subsequence and we are done.

iii) It is known that $B_{\infty}$ satisfies the convex point of
continuity property CPCP \cite{GM}, a weaker property than PCP. So
it is natural to ask weather a Banach space satisfies CPCP,
whenever every semi-normalized basic sequence has a boundedly
complete subsequence.\end{remark}

\end{document}